\documentclass[12pt]{article}
\usepackage{multicol}
\usepackage{stmaryrd}
\usepackage{a4}
\usepackage{amsthm}
\usepackage{amsmath}
\usepackage{amssymb}
\usepackage{amsfonts}
\usepackage{cite}
\usepackage{epsfig}
\usepackage{comment}
\usepackage{xcolor}
\usepackage{ifpdf}
\ifpdf
\fi

\usepackage{multicol}

\newtheorem{theorem}{Theorem}

\newtheorem{lemma}[theorem]{Lemma}

\newcommand\NN{{\mathbb N}}
\newcommand\RR{{\mathbb R}}
\newcommand\EE{{\mathbb E}}

\newcommand{\supp}[1]{\mbox{supp}\left(#1\right)}
\def\cA{{\mathcal A}}

\newcommand{\unlab}[2]{\left\llbracket #1\right\rrbracket_{#2}}
\DeclareMathOperator{\im}{Im}

\begin{document}
\title{Characterization of quasirandom permutations by a pattern sum\thanks{The first, second, fourth and sixth authors were supported by the European Research Council (ERC) under the European Union's Horizon 2020 research and innovation programme (grant agreement No 648509). The second and sixth authors were also supported by the MUNI Award in Science and Humanities of the Grant Agency of Masaryk university. The third author was supported by the Leverhulme Trust Early Career Fellowship ECF-2018-534. The fifth author was supported by the European Union's Horizon 2020 research and innovation programme under the Marie Curie grant agreement No 75242. This publication reflects only its authors' view; the European Research Council Executive Agency is not responsible for any use that may be made of the information it contains.}}
\author{Timothy F.~N. Chan\thanks{School of Mathematics, Monash University, Melbourne 3800, Australia, and Mathematics Institute, University of Warwick, Coventry CV4 7AL, UK. Email: {\tt timothy.chan@monash.edu}.}\and
        Daniel Kr\'al'\thanks{Faculty of Informatics, Masaryk University, Botanick\'a 68A, 602 00 Brno, Czech Republic, and Mathematics Institute, DIMAP and Department of Computer Science, University of Warwick, Coventry CV4 7AL, UK. E-mail: {\tt dkral@fi.muni.cz}.}\and
	Jonathan A. Noel\thanks{Mathematics Institute and DIMAP, University of Warwick, Coventry CV4 7AL, UK. E-mail: {\tt J.Noel@warwick.ac.uk}.}\and
	Yanitsa Pehova\thanks{Mathematics Institute, University of Warwick, Coventry CV4 7AL, UK. E-mail: {\tt Y.Pehova@warwick.ac.uk}.}\and
	Maryam Sharifzadeh\thanks{Mathematics Institute, University of Warwick, Coventry CV4 7AL, UK. E-mail: {\tt m.sharifzadeh@warwick.ac.uk}.}\and
        Jan Volec\thanks{Department of Mathematics, Faculty of Nuclear Sciences and Physical Engineering, Czech Technical University in Prague, Trojanova 13, 120 00 Prague, Czech Republic. Previous affiliation: Faculty of Informatics, Masaryk University, Botanick\'a 68A, 602 00 Brno, Czech Republic. E-mail: {\tt jan@ucw.cz}.}
}
\date{}
\maketitle
\begin{abstract}
It is known that a sequence $\{\Pi_i\}_{i\in\NN}$ of permutations is quasirandom
if and only if the pattern density of every $4$-point permutation in $\Pi_i$ converges to $1/24$.
We show that there is a set $S$ of $4$-point permutations such that
the sum of the pattern densities of the permutations from $S$ in the permutations $\Pi_i$ converges to $\lvert S\rvert/24$
if and only if the sequence is quasirandom. Moreover, we are able to completely characterize the sets $S$ with this property.
In particular, there are exactly ten such sets, the smallest of which has cardinality eight.
\end{abstract}

\section{Introduction}

A combinatorial object is said to be quasirandom
if it looks like a truly random object of the same kind.
The theory of \emph{quasirandom graphs} can be traced back to the work of R\"odl~\cite{Rod86}, Thomason~\cite{Tho87} and
Chung, Graham and Wilson~\cite{ChuGW89} from the 1980s.
It turned out that several diverse properties of random graphs
involving subgraph density, edge distribution and eigenvalues of the adjacency matrix
are satisfied by a large graph if and only if one of them is.
In particular, if the edge density of a large graph $G$ is $p+o(1)$ and
the density of cycles of length four is $p^4+o(1)$,
then the density of all subgraphs is close to their expected density in a random graph.
Results of similar kind have been obtained for many other types of combinatorial objects,
for example 
groups~\cite{Gow08},
hypergraphs~\cite{ChuG90,Gow06,Gow07,HavT89,KohRS02},
set systems~\cite{ChuG91s},
subsets of integers~\cite{ChuG92} and
tournaments~\cite{BucLSSXX,ChuG91,CorR17,HanKKMPSV19}.
In this paper, we will be concerned with quasirandomness of permutations as studied in~\cite{Coo04,KraP13}.

To state our results precisely, we need to fix some notation.
We use $[n]$ to denote the set $\{1,\ldots,n\}$. A \emph{permutation of order $k$}, 
or briefly a \emph{$k$-permutation}, is a bijection from $[k]$ to $[k]$.
The order of a permutation $\pi$ is denoted by $\lvert\pi\rvert$.
If $A=\{a_1,\ldots,a_{\ell}\}\subseteq [k]$, $a_1<\cdots<a_{\ell}$,
then the subpermutation $\pi$ \emph{induced} by $A$
is the unique permutation $\pi'$ of order $|A|=\ell$ such that 
$\pi'(i)<\pi'(j)$ if and only if $\pi(a_i)<\pi(a_j)$ for every $i,j\in [\ell]$.
Subpermutations are also often referred to as patterns.
If $\pi$ and $\Pi$ are two permutations, 
then the \emph{pattern density} of $\pi$ in $\Pi$, which is denoted by $d(\pi,\Pi)$,
is the probability that 
the subpermutation of $\Pi$ induced by a random $\lvert\pi\rvert$-element subset of $[n]$ is $\pi$,
where $n=\lvert\Pi\rvert$.
If $\lvert\Pi\rvert<\lvert\pi\rvert$, then we set $d(\pi,\Pi)$ to be zero.
We often refer to pattern density simply as \emph{density} in what follows. Finally,
a sequence $\{\Pi_i\}_{i\in\NN}$ of permutations is \emph{quasirandom} if 
\[\lim_{i\to\infty}d(\pi,\Pi_i)=\frac{1}{\lvert\pi\rvert!}\]
for every permutation $\pi$.

Our research is motivated by the following question of Graham (see \cite[page 141]{Coo04}):
\emph{Is there an integer $k$ such that a sequence $\{\Pi_i\}_{i\in\NN}$ of permutations is quasirandom if and only if
\[\lim_{i\to\infty}d(\pi,\Pi_i)=\frac{1}{k!}\]
for every $k$-permutation $\pi$?}
The question was answered affirmatively in~\cite{KraP13} by establishing that $k=4$ has this property.
It is interesting that
this statement is equivalent to a result in statistics on non-parametric independence tests by Yanagimoto~\cite{Yan70},
which improved an older result by Hoeffding~\cite{Hoe48}.
\begin{theorem}
\label{thm:gafa}
A sequence $\{\Pi_i\}_{i\in\NN}$ of permutations is quasirandom if and only if
\[\lim_{i\to\infty}d(\pi,\Pi_i)=\frac{1}{4!}\]
for every $4$-permutation $\pi$.
\end{theorem}
The statement of Theorem~\ref{thm:gafa} does not hold for $3$-permutations~\cite{CooP08}, also see~\cite{KraP13};
i.e., there exists a non-quasirandom sequence of permutations in which the density of every $3$-permutation converges to $1/3!$. 

Theorem~\ref{thm:gafa} says that
if the limit densities of \emph{all} $4$-permutations in a sequence are equal to $1/4!$,
then the sequence is quasirandom.
Hence, it is natural to ask
whether it is possible to replace the set of all $4$-permutations in the statement of Theorem~\ref{thm:gafa} with a smaller set.
Inspecting the proof given in~\cite{KraP13},
Zhang~\cite{Zha} observed that there exists a $16$-element set of $4$-permutations with this property.
We identify several $8$-element sets that have this property.
In fact, the sets $S$ that we identify have the stronger property that
fixing the sum of densities of elements of $S$ is enough to force quasirandomness; i.e. it 
is not necessary to fix the density of each individual element of $S$.
This stronger property was studied in statistics by Bergsma and Dassios~\cite{BerD14}
who also identified the first of the $8$-element sets listed in Theorem~\ref{thm:main} below.
Formally,
we say that a set $S$ of $k$-permutations is \emph{$\Sigma$-forcing} if the following holds:
a sequence $\{\Pi_i\}_{i\in\NN}$ of permutations is quasirandom if and only if
\[\lim_{i\to\infty}\sum_{\pi\in S}d(\pi,\Pi_i)=\frac{\lvert S\rvert}{k!}.\]
Our main theorem is the characterization of all $\Sigma$-forcing sets of $4$-permutations.

\begin{theorem}
\label{thm:main}
Let $S$ be a set of $4$-permutations.
The set $S$ is $\Sigma$-forcing if and only if
$S$ is one of the following five sets
\begin{itemize}
\item $\{1234,1243,2134,2143,3412,3421,4312,4321\}$,
\item $\{1234,1432,2143,2341,3214,3412,4123,4321\}$,
\item $\{1324,1342,2413,2431,3124,3142,4213,4231\}$,
\item $\{1324,1423,2314,2413,3142,3241,4132,4231\}$,
\item $\{1234,1243,1432,2134,2143,2341,3214,3412,3421,4123,4312,4321\}$, or
\end{itemize}
their complements.
\end{theorem}

Theorem~\ref{thm:main} is implied
by Theorems~\ref{thm:set8a}, \ref{thm:set8b}, \ref{thm:set8c}, \ref{thm:set12} and~\ref{thm:examples} 
which are stated and proved later in the paper (note that the third and fourth sets are rotationally symmetric).
Some of the arguments are supported by supplementary data,
which are presented in Appendices 1--5.
The Appendices are available as the ancillary file on arXiv,
which can be downloaded at {\tt https://arxiv.org/src/\\1909.11027/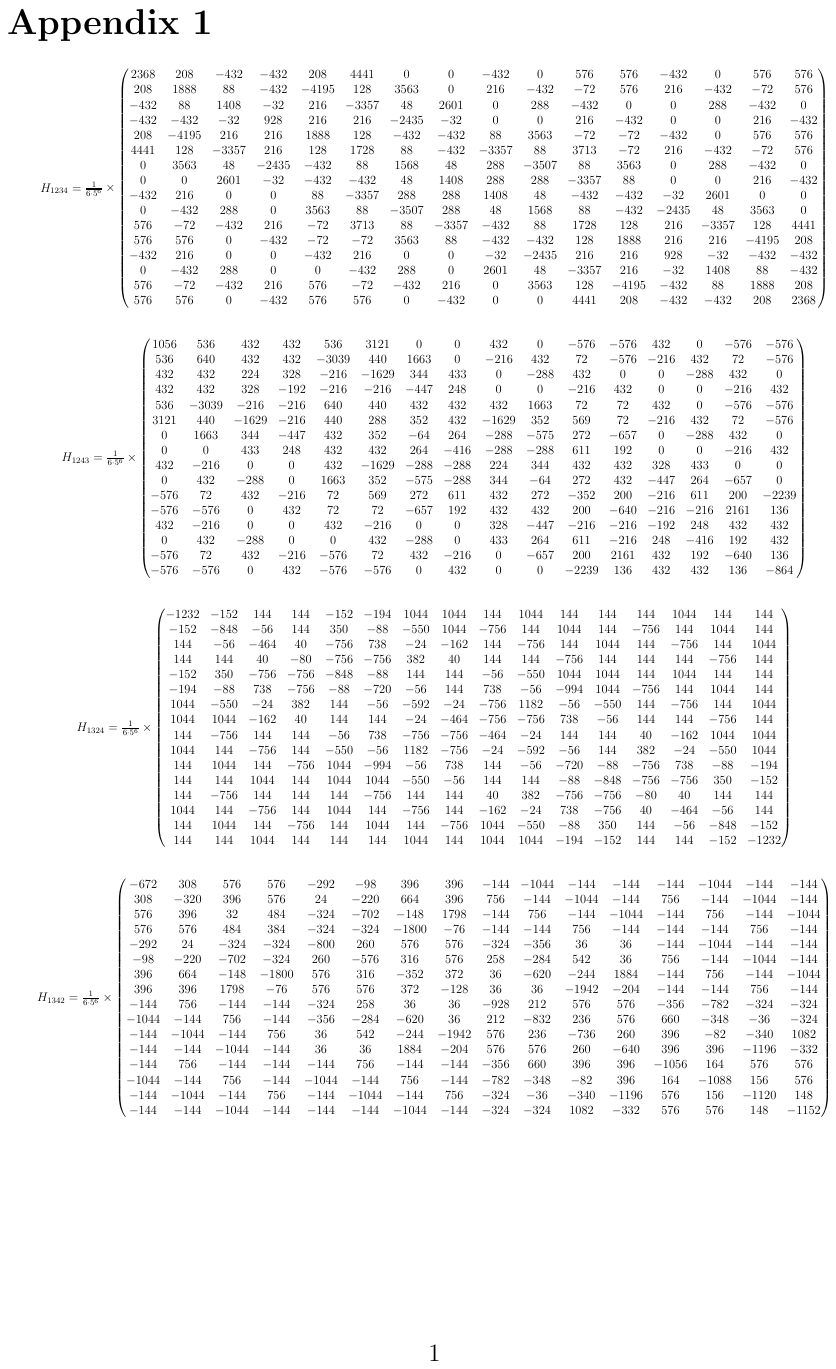}.

In relation to applications in statistics,
we remark that Even-Zohar and Leng~\cite{EveL19} designed nearly linear algorithms
for computing the sum of the pattern densities in an input permutation
for six out of the ten $\Sigma$-forcing sets listed in Theorem~\ref{thm:main}.

\section{Notation}
\label{sec:notation}

In this section, we fix the notation used throughout the paper.
The set of all $k$-permutations is denoted by $S_k$, and
$\cA$ denotes the set of all formal (finite) linear combinations of permutations with real coefficients.
If $\pi$ is a $k$-permutation, we often write $\pi(1)\pi(2)\ldots\pi(k)$ to represent the permutation $\pi$;
for example, $32145$ is a particular $5$-permutation.
Two permutations $\pi$ and $\sigma$ of the same order, say $k$, are \emph{symmetric}
if the permutation matrix of $\pi$ can be obtained from the permutation matrix of $\sigma$
by a sequence of reflections and rotations when viewed as $k\times k$ tables,
i.e., if either
\begin{multicols}{2}
\begin{itemize}
\item $\pi(i)=\sigma(i)$, or
\item $\pi(i)=\sigma(k+1-i)$, or
\item $\pi(i)=k+1-\sigma(i)$, or
\item $\pi(i)=k+1-\sigma(k+1-i)$, or
\item $\pi(i)=\sigma^{-1}(i)$, or
\item $\pi(i)=\sigma^{-1}(k+1-i)$, or
\item $\pi(i)=k+1-\sigma^{-1}(i)$, or
\item $\pi(i)=k+1-\sigma^{-1}(k+1-i)$ 
\end{itemize}
\end{multicols}
\noindent holds for all $i\in [k]$.
For example, exactly the following seven permutations are symmetric to $12534$ in addition to the permutation $12534$ itself:
$12453$, $23145$, $31245$, $35421$, $43521$, $54132$ and $54213$.

If $\tau$ is a permutation, then a \emph{$\tau$-rooted permutation}
is a permutation with $\lvert\tau\rvert$ distinguished elements such that 
the subpermutation induced by these elements is $\tau$;
the distinguished elements are referred to as \emph{roots}.
The two particular choices of a root permutation that we work with most often are $\tau_1=12$ and $\tau_2=21$.
When presenting rooted permutations, the roots will be underlined.
For example, 
$\underline{12}3$, $1\underline{23}$ and $\underline{1}2\underline{3}$ are distinct $\tau_1$-rooted permutations.
Finally,
the set $\cA^\tau$ is the set of all formal (finite) linear combinations of $\tau$-rooted permutations with real coefficients.

A \emph{permuton} is a Borel probability measure $\mu$ on $[0,1]^2$ that has uniform marginals,
i.e., $\mu\left([x,x']\times [0,1]\right)=x'-x$ for every $0\le x<x'\le 1$ and
$\mu\left([0,1]\times [y,y']\right)=y'-y$ for every $0\le y<y'\le 1$.
In other contexts, permutons are known as doubly stochastic measures or two-dimensional copulas. 
Given a permuton $\mu$, a \emph{$\mu$-random permutation of order $k$} is obtained in the way that we now describe.
We first sample $k$ points $(x_1,y_1),\ldots,(x_k,y_k)$ in $[0,1]^2$ according to the probability measure $\mu$.
Note that the probability that an $x$- or $y$-coordinate is shared by multiple points is zero
because $\mu$ has uniform marginals.
By renaming the points, we can assume that $x_1<\cdots<x_k$.
The $\mu$-random permutation $\pi\in S_k$ is then the unique permutation such that $\pi(i)<\pi(j)$ if and only if $y_i<y_j$ for every $i,j\in [k]$.
We define the \emph{pattern density} of $\pi\in S_k$ in the permuton $\mu$ to be the probability that
a $\mu$-random permutation of order $k$ is $\pi$.
A sequence $(\Pi_i)_{i\in\NN}$ of permutations is \emph{convergent}
if $\lvert\Pi_i\rvert$ grows to infinity and
the limit
\[\lim_{i\to\infty}d(\pi,\Pi_i)\]
exists for every permutation $\pi$.
It can be shown~\cite{HopKMRS13,HopKMS11a,KraP13} that
if $(\Pi_i)_{i\in\NN}$ is a convergent sequence of permutations, then there exists a unique permuton $\mu$ such that
\[\lim_{i\to\infty}d(\pi,\Pi_i)=d(\pi,\mu)\]
for every permutation $\pi$;
the permuton $\mu$ is called the \emph{limit} of the sequence $(\Pi_i)_{i\in\NN}$.
In the other direction, if $\mu$ is a permuton,
then, with probability one, a sequence of $\mu$-random permutations with increasing orders converges and
its limit is $\mu$.
We note that a sequence $(\Pi_i)_{i\in\NN}$ of permutations is quasirandom if and only if
its limit is the uniform measure on $[0,1]^2$.

Recall that the \emph{support} of a Borel measure $\mu$, denoted $\supp\mu$, is the set of all points $x$ such that
every open neighborhood of $x$ has positive measure under $\mu$. 
Fix a permutation $\tau\in S_\ell$.
A \emph{$\tau$-rooted permuton} is an $(\ell+1)$-tuple $\mu^\tau=(\mu,(x_1,y_1),\ldots,(x_\ell,y_\ell))$ such that
\begin{itemize}
\item $\mu$ is a permuton,
\item $(x_1,y_1),\ldots,(x_\ell,y_\ell)\in\supp\mu$, $x_1<\cdots<x_\ell$, and
\item $\tau(i)<\tau(j)$ if and only if $y_i<y_j$ for all $i,j\in [\ell]$.
\end{itemize}
The points $(x_1,y_1),\ldots,(x_\ell,y_\ell)$ are referred to as \emph{roots}.
If $\mu^\tau$ is a $\tau$-rooted permuton, then a \emph{$\mu^\tau$-random permutation of order $k\ge\ell$}
is a $\tau$-rooted permutation obtained by sampling $k-\ell$ points in $[0,1]^2$ according to the measure $\mu$,
forming a permutation of order $k$ using the $\ell$ roots and the $k-\ell$ sampled points, and
distinguishing the $\ell$ points corresponding to the roots of $\mu^\tau$ to be the roots of the permutation.
If $\pi^\tau$ is a $\tau$-rooted permutation, we write $d(\pi^\tau,\mu^\tau)$ for the probability that
a $\mu^\tau$-random permutation of order $\lvert\pi^\tau\rvert$ is $\pi^\tau$.

Fix a permuton $\mu$ for the rest of this section.
We define a mapping $h_{\mu}:\cA\to\RR$
by setting $h_{\mu}(\pi)$ to be $d(\pi,\mu)$ for every permutation $\pi$ and extending linearly.
Clearly, $h_{\mu}$ is a homomorphism from $\cA$ to $\RR$ that respects addition and multiplication by a real number.
One of the results of Razborov~\cite{Raz07} can be cast in our setting as follows:
it is possible to define a multiplication on the elements of $\cA$ in a way that
$h_{\mu}$ respects the multiplication, i.e., $h_{\mu}(A\times B)=h_{\mu}(A)h_{\mu}(B)$ for all $A,B\in\cA$.
We write $A\ge\alpha$ for an element $A\in\cA$ and a real $\alpha\in\RR$
if $h_{\mu}(A)\ge\alpha$ for every permuton $\mu$.
Analogously to the unrooted case,
for a $\tau$-rooted permuton $\mu^\tau$, we can define a homomorphism $h_{\mu^\tau}:\cA^\tau\to\RR$.

Next, for a permutation $\tau$ with $d(\tau,\mu)>0$,
we wish to define a probability distribution on $\tau$-rooted permutons arising from $\mu$.
Formally, we define $\mu^\tau$ to be a $\tau$-rooted permuton obtained from $\mu$
by choosing $\lvert\tau\rvert$ points randomly according to the probability measure $\mu$ to be the roots (and
sorting them according to their first coordinates)
conditioned on the event that the chosen roots yield the permutation $\tau$,
i.e., $\mu^\tau$ is a random $\tau$-rooted permuton
where the randomness comes from the choice of $\lvert\tau\rvert$ roots.
The probability distribution on $\tau$-rooted permutons
in turn defines a probability distribution on homomorphisms from $\cA^\tau$ to $\RR$, and
we will write $h_{\mu}^{\tau}$ for a random homomorphism from $\cA^\tau$ to $\RR$ sampled according to this distribution.
It can be shown~\cite{Raz07} that
there exists a well-defined linear map $\unlab{\cdot}{\tau}$ from $\cA^\tau$ to $\cA$ such that
\[h_{\mu}\left(\unlab{A}{\tau}\right)=d(\tau,\mu)\cdot\EE h_{\mu}^\tau(A)\]
for every $A\in\cA^\tau$;
if $d(\tau,\mu)=0$,
then $h_{\mu}\left(\unlab{A}{\tau}\right)=0$ and
the above equality holds with the right hand side considered to be zero (the expected
value is not well-defined in this case).
In particular, if $M$ is a $k\times k$ positive semidefinite matrix,
then the following holds for every vector $w\in\left(\cA^\tau\right)^k$:
\[h_{\mu}\left(\unlab{w^T Mw}{\tau}\right) \ge 0.\]

\section{$\Sigma$-forcing sets}
\label{sec:flags}

In this section, we prove that the sets listed in Theorem~\ref{thm:main} are $\Sigma$-forcing.
The proof is based on flag algebra calculations, which we present further in the section.
In addition, we will need the following lemma.

\begin{lemma}
\label{lm:supp}
Let $\mu$ be a permuton.
If it holds that
\[\mu\left(\left[\min\{x_1,x_2\},\max\{x_1,x_2\}\right]\times\left[\min\{y_1,y_2\},\max\{y_1,y_2\}\right]\right)=\lvert x_2-x_1\rvert\cdot\lvert y_2-y_1\rvert\]
for all points $(x_1,y_1),(x_2,y_2)\in\supp\mu$,
then $\mu$ is the uniform measure.
\end{lemma}

\begin{proof}
Our goal is to show that $\supp\mu=[0,1]^2$.
We start with showing that all points on the boundary of $[0,1]^2$ are contained in $\supp\mu$.
Suppose that $\supp\mu$ does not contain the whole boundary of $[0,1]^2$.
Since $\supp\mu$ is closed, it is enough to consider the points distinct from the four corners.
By symmetry, we need to consider the following two cases.
\begin{itemize}
\item{\bf There exists $x\in (0,1)$ such that $(x,0)\not\in\supp\mu$ but $(x,1)\in\supp\mu$.}
     By the definition of the support of a measure,
     there exists $\varepsilon\in(0,\min\{x,1-x\})$ such that
     \[\mu\left([x-\varepsilon,x+\varepsilon]\times[0,\varepsilon]\right)=0.\]
     Let $y'\in [0,1]$ be the infimum among all reals such that
     $(x',y')\in\supp\mu$ for some $x'\in (x-\varepsilon,x)$.
     If there was no such $y'$,
     then the measure of the rectangle $[x-\varepsilon,x]\times [0,1]$ would be zero,
     which is impossible because the measure $\mu$ has uniform marginals.
     Observe that $y'\in [\varepsilon,1]$.
     Since $\supp\mu$ is a closed set, there exists $x'\in [x-\varepsilon,x]$ such that $(x',y')\in\supp\mu$;
     if possible, choose $x'$ distinct from $x$.
     
     We first consider the case that $x'<x$.
     The assumption of the lemma implies that
     the measure of the rectangle $[x',x]\times [y',1]$ is $(x-x')(1-y')$.
     On the other hand, the choice of $y'$ implies that
     the measure of the rectangle $[x',x]\times [0,y']$ is zero.
     Consequently, the measure of the rectangle $[x',x]\times [0,1]$ is $(x-x')(1-y')<x-x'$,
     which is impossible.

     It remains to analyze the case that $x'=x$.
     The choice of $y'$ implies that there exist $y''\in (y',1]$ and $x''\in (x-\varepsilon,x)$
     such that $(x'',y'')\in\supp\mu$.
     Since the measure of the rectangle $[x'',x]\times [y'',1]$ is $(x-x'')(1-y'')$ and
     the measure of the rectangle $[x'',x]\times [y',y'']$ is $(x-x'')(y''-y')$,
     the measure of the rectangle $[x'',x]\times [y',1]$ is $(x-x'')(1-y')$.
     On the other hand, the choice of $y'$ implies that
     the measure of the rectangle $[x'',x]\times [0,y']$ is zero,
     which yields that the measure of the rectangle $[x'',x]\times [0,1]$ is less than $x-x''$,
     which is impossible.
\item{\bf There exists $x\in (0,1)$ such that $(x,0)\not\in \supp{\mu}$ and $(x,1)\not\in \supp\mu$.}
     By the definition of the support of a measure,
     there exists $\varepsilon\in(0,\min\{x,1-x\})$ such that
     \[\mu\left([x-\varepsilon,x+\varepsilon]\times[0,\varepsilon]\right)=0\mbox{ and }
       \mu\left([x-\varepsilon,x+\varepsilon]\times[1-\varepsilon,1]\right)=0.\]
     Let $y_1\in [0,1]$ be the infimum among all reals such that
     $(x_1,y_1)\in\supp\mu$ for some $x_1\in (x-\varepsilon,x+\varepsilon)$.
     If there was no such $y_1$,
     then the measure of the rectangle $[x-\varepsilon,x+\varepsilon]\times [0,1]$ would be zero,
     which is impossible because the measure $\mu$ has uniform marginals.
     Since $\supp\mu$ is a closed set, there exists $x_1\in [x-\varepsilon,x+\varepsilon]$ such that $(x_1,y_1)\in\supp\mu$.
     Note that $y_1\in [\varepsilon,1-\varepsilon]$.
     Similarly,
     let $y_2\in [0,1]$ be the supremum among all reals such that
     $(x_2,y_2)\in\supp\mu$ for some $x_2\in (x-\varepsilon,x+\varepsilon)$ (again note that $y_2\in [\varepsilon,1-\varepsilon]$) and
     we fix $x_2\in [x-\varepsilon,x+\varepsilon]$ such that $(x_2,y_2)\in\supp\mu$.
     If possible, we choose $x_1$ and $x_2$ above such that $x_1\ne x_2$.

     We first consider the case that $x_1\ne x_2$; by symmetry, we can assume that $x_1<x_2$.
     The assumption of the lemma implies that the measure of the rectangle $[x_1,x_2]\times [y_1,y_2]$ is $(x_2-x_1)(y_2-y_1)$, and
     the choices of $y_1$ and $y_2$ imply that
     the measure of each of the rectangles $[x_1,x_2]\times [0,y_1]$ and $[x_1,x_2]\times [y_2,1]$ is zero.
     It follows that the measure of the rectangle $[x_1,x_2]\times [0,1]$ is $(x_2-x_1)(y_2-y_1)<x_2-x_1$,
     which is impossible.

     It remains to consider the case that $x_1=x_2$.
     Since the measure of the rectangle $[x-\varepsilon,x+\varepsilon]\times [0,1]$ is not zero,
     there exists $x_3\in [x-\varepsilon,x+\varepsilon]$, $x_3\ne x_1$, and $y_3\in (y_1,y_2)$ such that $(x_3,y_3)\in\supp\mu$.
     By symmetry, we can assume that $x_1<x_3$.
     The measures of the rectangles $[x_1,x_3]\times [y_1,y_3]$ and $[x_1,x_3]\times [y_3,y_2]$
     are $(x_3-x_1)(y_3-y_1)$ and $(x_3-x_1)(y_2-y_3)$, respectively.
     Since the measure of each of the rectangles $[x_1,x_3]\times [0,y_1]$ and $[x_1,x_3]\times [y_2,1]$ is zero,
     we conclude that the measure of the rectangle $[x_1,x_3]\times [0,1]$ is $(x_3-x_1)(y_2-y_1)<x_3-x_1$,
     which is impossible.
\end{itemize}

We have shown that all points on the boundary of $[0,1]^2$ are contained in $\supp\mu$.
Suppose that there exists a point $(x,y)\in (0,1)^2$ that is not contained in $\supp\mu$, and
let $\varepsilon \in \left(0, \min\{x,y,1-x,1-y\}\right)$ be such that
the whole set $[x-\varepsilon,x+\varepsilon]\times [y-\varepsilon,y+\varepsilon]$ is not contained in $\supp\mu$.
Let $y_1$ be the supremum among all reals in $[0,y-\varepsilon]$ such that
$(x_1,y_1)\in\supp\mu$ for some $x_1\in (x-\varepsilon,x+\varepsilon)$, and
let $y_2$ be the infimum among all reals in $[y+\varepsilon,1]$ such that
$(x_2,y_2)\in\supp\mu$ for some $x_2\in (x-\varepsilon,x+\varepsilon)$.
Further, let $x_1,x_2\in [x-\varepsilon,x+\varepsilon]$ be such that
$(x_1,y_1)\in\supp\mu$ and $(x_2,y_2)\in\supp\mu$.
Note that $y_1$ can be $0$ and $y_2$ can be $1$, and $y_2-y_1\ge 2\varepsilon$.

We first consider the case that $x_1\ne x_2$. By symmetry, we can assume that $x_1<x_2$.
Since the boundary of the square $[0,1]^2$ is contained in $\supp\mu$,
the measures of the rectangles $[x_1,x_2]\times [0,y_1]$ and $[x_1,x_2]\times [y_2,1]$
are $(x_2-x_1)y_1$ and $(x_2-x_1)(1-y_2)$, respectively.
On the other hand, the choice of $y_1$ and $y_2$ implies that
the measure of the rectangle $[x_1,x_2]\times [y_1,y_2]$ is zero.
Consequently, the measure of the rectangle $[x_1,x_2]\times [0,1]$ is $(x_2-x_1)(1-y_2+y_1)<x_2-x_1$,
which is impossible.

To conclude the proof, we need to analyze the case $x_1=x_2$.
Let $x_3$ be any point in the interval $[x-\varepsilon,x+\varepsilon]$ distinct from $x_1=x_2$.
By symmetry, we can assume that $x_1<x_3$.
Again, since the boundary of the square $[0,1]^2$ is contained in $\supp\mu$,
it follows that 
the measures of the rectangles $[x_1,x_3]\times [0,y_1]$ and $[x_1,x_3]\times [y_2,1]$
are $(x_3-x_1)y_1$ and $(x_3-x_1)(1-y_2)$, respectively, and
the choice of $y_1$ and $y_2$ yields that
the measure of the rectangle $[x_1,x_3]\times [y_1,y_2]$ is zero.
We obtain that the measure of the rectangle $[x_1,x_3]\times [0,1]$ is $(x_3-x_1)(1-y_2+y_1)<x_3-x_1$,
which is impossible.
We can now conclude that the support of the measure $\mu$ is the whole square $[0,1]^2$.
Consequently the measure of each set $[x,x']\times [y,y']$ is equal to $(x'-x)(y'-y)$, 
which implies that the measure $\mu$ is the uniform measure on $[0,1]^2$.
This finishes the proof of the lemma.
\end{proof}

For the rest of the section,
we fix the following elements $A_1\in\cA^{\tau_1}$ and $A_2\in\cA^{\tau_2}$.
\begin{align*}
A_1 & = \left(\underline{1}2\underline{3}4 - \underline{1}4\underline{3}2\right) + \left(1\underline{2}3\underline{4} - 3\underline{2}1\underline{4}\right) + \left(\underline{2}3\underline{4}1 - \underline{2}1\underline{4}3\right) + \left(4\underline{1}2\underline{3} - 2\underline{1}4\underline{3}\right)\\
A_2 & = \left(\underline{3}2\underline{1}4 - \underline{3}4\underline{1}2\right) + \left(1\underline{4}3\underline{2} - 3\underline{4}1\underline{2}\right) + \left(\underline{4}3\underline{2}1 - \underline{4}1\underline{2}3\right) + \left(4\underline{3}2\underline{1} - 2\underline{3}4\underline{1}\right)
\end{align*}
We next show that if the value of $A_i$ is zero for almost all random $\tau_1$-rooted homomorphisms for both $i=1$ and $i=2$,
then the permuton satisfies the assumptions of Lemma~\ref{lm:supp}.
\begin{lemma}
\label{lm:A1}
Let $\mu$ be a permuton.
If $h_{\mu}^{\tau_1}(A_1)=0$ with probability one,
then
\[\mu\left(\left[x_1,x_2\right]\times\left[y_1,y_2\right]\right)=\lvert x_2-x_1\rvert\cdot\lvert y_2-y_1\rvert\]
for all points $(x_1,y_1),(x_2,y_2)\in\supp\mu$ such that $x_1\le x_2$ and $y_1\le y_2$.
\end{lemma}

\begin{figure}
\begin{center}
\epsfbox{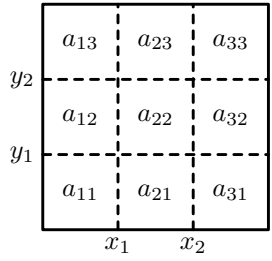}
\end{center}
\caption{Notation used in the proof of Lemma~\ref{lm:A1}.}
\label{fig:A1}
\end{figure}

\begin{proof}
Fix $(x_1,y_1),(x_2,y_2)\in\supp\mu$ such that $x_1\le x_2$ and $y_1\le y_2$ and
such that $h(A_1)=0$ for the homomorphism $h:\cA^{\tau_1}\to\RR$
associated with the $\tau_1$-rooted permuton $(\mu,(x_1,y_1),(x_2,y_2))$.
Further let $(x_0,y_0)=(0,0)$ and $(x_3,y_3)=(1,1)$, and
let
\[a_{ij}=\mu\left([x_{i-1},x_i]\times [y_{j-1},y_j]\right)\]
for $i,j\in [3]$.
See Figure~\ref{fig:A1} for illustration of the just introduced notation.

Since $h(A_1)=0$, the following holds:
\[a_{22}a_{33}-a_{23}a_{32}+a_{22}a_{11}-a_{12}a_{21}+a_{22}a_{31}-a_{21}a_{32}+a_{22}a_{13}-a_{12}a_{23}=0\;.\]
We rewrite this expression using the property that $\mu$ has uniform marginals as follows:
\begin{align*}
0 & = a_{22}a_{33}-a_{23}a_{32}+a_{22}a_{11}-a_{12}a_{21}+a_{22}a_{31}-a_{21}a_{32}+a_{22}a_{13}-a_{12}a_{23} \\
  & = a_{22}(a_{11}+a_{13}+a_{31}+a_{33})-(a_{21}+a_{23})(a_{12}+a_{32}) \\
  & = a_{22}(1-(x_2-x_1)-(y_2-y_1)+a_{22})-(x_2-x_1-a_{22})(y_2-y_1-a_{22}) \\
  & = a_{22}-(x_2-x_1)(y_2-y_1)
\end{align*}
We conclude that the equality from the statement of the lemma holds
for almost all points $(x_1,y_1),(x_2,y_2)\in\supp\mu$ such that $x_1\le x_2$ and $y_1\le y_2$.

We next show that the equality in the statement of the lemma holds
for all $(x_1,y_1),(x_2,y_2)\in\supp\mu$ such that $x_1\le x_2$ and $y_1\le y_2$.
Fix $(x_1,y_1),(x_2,y_2)\in\supp\mu$ such that $x_1\le x_2$ and $y_1\le y_2$.
If $x_1=x_2$ or $y_1=y_2$, then the equality holds since the measure $\mu$ has uniform marginals.
Let $\varepsilon_0=\min\{x_2-x_1,y_2-y_1\}$ and consider $\varepsilon\in (0,\varepsilon_0/2)$.
Since all points in the $\varepsilon$-neighborhood of $(x_1,y_1)$
have both their coordinates smaller than all points in the $\varepsilon$-neighborhood of $(x_2,y_2)$,
almost every point $(x'_1,y'_1)$ in the intersection of $\supp\mu$ and the $\varepsilon$-neighborhood of $(x_1,y_1)$ and
almost every point $(x'_2,y'_2)$ in the intersection of $\supp\mu$ and the $\varepsilon$-neighborhood of $(x_2,y_2)$
satisfy the equality from the statement of the lemma, and
it also holds that
\[\left|\mu\left(\left[x_1,x_2\right]\times\left[y_1,y_2\right]\right)-
        \mu\left(\left[x'_1,x'_2\right]\times\left[y'_1,y'_2\right]\right)\right|\le 4\varepsilon\]
because the measure $\mu$ has uniform marginals.	
Since both $(x_1,y_1)$ and $(x_2,y_2)$ are contained in $\supp\mu$,
the $\varepsilon$-neighborhood of $(x_1,y_1)$ has positive measure and
the $\varepsilon$-neighborhood of $(x_2,y_2)$ also has positive measure,
we conclude that
\[\big\lvert\mu\left(\left[x_1,x_2\right]\times\left[y_1,y_2\right]\right)-\lvert x_2-x_1\rvert\cdot\lvert y_2-y_1\rvert\big\rvert\le 8\varepsilon\]
for every $\varepsilon\in (0,\varepsilon_0/2)$.
It follows that the equality from the statement of the lemma holds
for all points $(x_1,y_1),(x_2,y_2)\in\supp\mu$ such that $x_1\le x_2$ and $y_1\le y_2$.
\end{proof}

A symmetric argument yields the following lemma.

\begin{lemma}
\label{lm:A2}
Let $\mu$ be a permuton.
If $h_{\mu}^{\tau_2}(A_2)=0$ with probability one,
then
\[\mu\left(\left[x_1,x_2\right]\times\left[y_1,y_2\right]\right)=\lvert x_2-x_1\rvert\cdot\lvert y_2-y_1\rvert\]
for all points $(x_1,y_2),(x_2,y_1)\in\supp\mu$ such that $x_1\le x_2$ and $y_1\le y_2$.
\end{lemma}

We are now ready to prove the first of the four main results of this section.
The proofs of Theorems~\ref{thm:set8a}, \ref{thm:set8b}, \ref{thm:set8c} and \ref{thm:set12} are based on the flag algebra method.
We follow the standard path of applying the method by setting up appropriate SDP programs;
solving these programs yields the positive semidefinite matrices $M$ and vectors $A_1,\ldots,O_1$ and $A_2,\ldots,O_2$
used in the proofs of the four theorems.

\begin{theorem}
\label{thm:set8a}
Let $S=\{1234,1243,2134,2143,3412,3421,4312,4321\}$.
It holds that
\[\sum_{\pi\in S}d(\pi,\mu)\ge\frac{1}{3}\]
for every permuton $\mu$, and equality holds if and only if $\mu$ is uniform.
\end{theorem}

\begin{proof}
Let $B_1$, $C_1$, $D_1$ and $E_1$ be the following four elements of $\cA^{\tau_1}$.
\begin{align*}
B_1 &= \left(1\underline{2}3\underline{4} - 3\underline{2}1\underline{4}\right) + \left(1\underline{23}4 - 4\underline{23}1\right)
     + \left(1\underline{24}3 - 3\underline{24}1\right) + \left(1\underline{2}4\underline{3} - 4\underline{2}1\underline{3}\right)\\
C_1 &= \left(\underline{1}2\underline{3}4 - \underline{1}4\underline{3}2\right) + \left(1\underline{23}4 - 4\underline{23}1\right)
     + \left(\underline{2}1\underline{3}4 - \underline{2}4\underline{3}1\right) + \left(2\underline{13}4 - 4\underline{13}2\right)\\
D_1 &= \left(2\underline{1}4\underline{3} - 4\underline{1}2\underline{3}\right) + \left(1\underline{23}4 - 4\underline{23}1\right)
     + \left(2\underline{13}4 - 4\underline{13}2\right) + \left(1\underline{2}4\underline{3} - 4\underline{2}1\underline{3}\right)\\
E_1 &= \left(\underline{2}1\underline{4}3 - \underline{2}3\underline{4}1\right) + \left(1\underline{23}4 - 4\underline{23}1\right)
     + \left(\underline{2}1\underline{3}4 - \underline{2}4\underline{3}1\right) + \left(1\underline{24}3 - 3\underline{24}1\right)
\end{align*}
Further, let $B_2$, $C_2$, $D_2$ and $E_2$ be the corresponding four elements of $\cA^{\tau_2}$,
e.g., $B_2$ is the following element:
\[ B_2 = \left(1\underline{4}3\underline{2} - 3\underline{4}1\underline{2}\right) + \left(1\underline{32}4 - 4\underline{32}1\right)
       + \left(1\underline{42}3 - 3\underline{42}1\right) + \left(1\underline{3}4\underline{2} - 4\underline{3}1\underline{2}\right).\]
Finally, let $M$ be the following (positive definite) matrix.
\[M=\begin{bmatrix}
1&0&0&0\\
0&1&0&0\\
0&0&2&0\\
0&0&0&2
\end{bmatrix}\]
A direct computation yields that
\begin{align*}
h_{\mu}\left(\unlab{w_1Mw_1^T}{\tau_1}+\unlab{w_2Mw_2^T}{\tau_2}\right)
 & = h_{\mu}\left(\frac{8}{9}\sum_{\pi\in S}\pi-\frac{2}{9}\sum_{\pi\in S_4\setminus S}\pi\right)\\
 & = \frac{2}{3}\left(\sum_{\pi\in S}d(\pi,\mu)-\frac{1}{3}\right)
\end{align*}
where $w_1=(B_1,C_1,D_1,E_1)$ and $w_2=(B_2,C_2,D_2,E_2)$.
Since the matrix $M$ is positive semidefinite,
it holds that $h_{\mu}\left(\unlab{w_1Mw_1^T}{\tau_1}\right)\ge 0$ and
$h_{\mu}\left(\unlab{w_2Mw_2^T}{\tau_2}\right)\ge 0$,
which implies that
\[0\le \sum_{\pi\in S}d(\pi,\mu)-\frac{1}{3}.\]
Moreover, the equality holds if and only if
both $h_{\mu}^{\tau_1}(w_1Mw_1^T)=0$ with probability one and
$h_{\mu}^{\tau_2}(w_2Mw_2^T)=0$ with probability one.
Since all the eigenvalues of the matrix $M$ are positive,
$h_{\mu}^{\tau_1}(w_1Mw_1^T)=0$ if and only if
$h_{\mu}^{\tau_1}(B_1)=0$, $h_{\mu}^{\tau_1}(C_1)=0$, $h_{\mu}^{\tau_1}(D_1)=0$ and $h_{\mu}^{\tau_1}(E_1)=0$.
Since $A_1=B_1+C_1-D_1-E_1$, we conclude that
if the equality holds, then $h_{\mu}^{\tau_1}(A_1)=0$ with probability one.
A symmetric argument yields that
if the equality holds, then $h_{\mu}^{\tau_2}(A_2)=0$ with probability one.
The statement of the theorem now follows from Lemmas~\ref{lm:supp}, \ref{lm:A1} and \ref{lm:A2}.
\end{proof}

We next prove the second main theorem of this section.
Since the proofs of this theorem and the two subsequent to it are similar to the proof of Theorem~\ref{thm:set8a},
we will be brief in their parts that are analogous.

\begin{theorem}
\label{thm:set8b}
Let $S=\{1234,1432,2143,2341,3214,3412,4123,4321\}$.
It holds that
\[\sum_{\pi\in S}d(\pi,\mu)\ge\frac{1}{3}\]
for every permuton $\mu$, and equality holds if and only if $\mu$ is uniform.
\end{theorem}

\begin{proof}
Consider the following elements $F_1$ and $G_1$ of $\cA^{\tau_1}$.
\begin{align*}
F_1 & = \left(1\underline{24}3 - 3\underline{24}1\right) + \left(4\underline{13}2 - 2\underline{13}4\right) + \left(\underline{1}24\underline{3} - \underline{1}42\underline{3}\right)+ \left(\underline{2}31\underline{4} - \underline{2}13\underline{4}\right)\\
    & + \left(\underline{13}24 - \underline{13}42\right)+ \left(\underline{24}31 - \underline{24}13\right)+ \left(31\underline{24} - 13\underline{24}\right) + \left(24\underline{13} - 42\underline{13}\right)\\
G_1 & = \left(\underline{12}43 - \underline{12}34\right)+ \left(\underline{34}21 - \underline{34}12\right) + \left(\underline{14}32 - \underline{14}23\right)+ \left(\underline{23}14 - \underline{23}41\right)\\
    & + \left(43\underline{12} - 34\underline{12}\right)+ \left(21\underline{34} - 12\underline{34}\right)+ \left(32\underline{14} - 23\underline{14}\right) +\left(14\underline{23} - 41\underline{23}\right)\\
    & +\left(\underline{1}43\underline{2} - \underline{1}34\underline{2}\right) +\left(\underline{3}21\underline{4} - \underline{3}12\underline{4}\right)+\left(\underline{1}32\underline{4} - \underline{1}23\underline{4}\right)+ \left(\underline{2}14\underline{3} - \underline{2}41\underline{3}\right)\\
    & + \left(3\underline{12}4 - 4\underline{12}3\right)+ \left(1\underline{34}2 - 2\underline{34}1\right) + \left(2\underline{14}3 - 3\underline{14}2\right)+ \left(4\underline{23}1 - 1\underline{23}4\right)
\end{align*}    
Let $F_2$ and $G_2$ be the corresponding elements of $\cA^{\tau_2}$ as in the proof of Theorem~\ref{thm:set8a}, and
let $M$ be the following (positive definite) matrix.
\[M=\begin{bmatrix}
5&0&3\\
0&9&0\\
3&0&4
\end{bmatrix}\]
It holds that
\[h_{\mu}\left(\unlab{w_1Mw_1^T}{\tau_1}+\unlab{w_2Mw_2^T}{\tau_2}\right)=2\left(\sum_{\pi\in S}d(\pi,\mu)-\frac{1}{3}\right)\]
where $w_1=(A_1,F_1,G_1)$ and $w_2=(A_2,F_2,G_2)$.
This implies that
\[0\le\sum_{\pi\in S}d(\pi,\mu)-\frac{1}{3}\]
and
 equality holds if and only if
both $h_{\mu}^{\tau_1}(w_1Mw_1^T)=0$ with probability one and
$h_{\mu}^{\tau_2}(w_2Mw_2^T)=0$ with probability one.
Since all the eigenvalues of $M$ are positive (the eigenvalues are $9$ and $\frac{9\pm\sqrt{37}}{2}$),
it follows that equality holds if and only if
$h_{\mu}^{\tau_1}(A_1)=0$ with probability one and $h_{\mu}^{\tau_2}(A_2)=0$ with probability one.
The statement of the theorem now follows from Lemmas~\ref{lm:supp}, \ref{lm:A1} and \ref{lm:A2}.
\end{proof}

We next prove the third main theorem of this section.

\begin{theorem}
\label{thm:set8c}
Let $S=\{1324,1342,2413,2431,3124,3142,4213,4231\}$.
It holds that
\[\sum_{\pi\in S}d(\pi,\mu)\le\frac{1}{3}\]
for every permuton $\mu$, and equality holds if and only if $\mu$ is uniform.
\end{theorem}

\begin{proof}
Let $\overline{S}=S_4\setminus S$ and consider the following four elements of $\cA^{\tau_1}$.
\begin{align*}
H_1 & = \left(1\underline{2}3\underline{4} - 3\underline{2}1\underline{4}\right) + \left(\underline{2}3\underline{4}1 - \underline{2}1\underline{4}3\right) +\left(1\underline{2}4\underline{3} - 4\underline{2}1\underline{3}\right) + \left(\underline{2}4\underline{3}1 - \underline{2}1\underline{3}4\right)\\
I_1 & = \left(2\underline{1}4\underline{3} - 4\underline{1}2\underline{3}\right) + \left(\underline{1}4\underline{3}2 - \underline{1}2\underline{3}4\right)+\left(1\underline{2}4\underline{3} - 4\underline{2}1\underline{3}\right) + \left(\underline{2}4\underline{3}1 - \underline{2}1\underline{3}4\right)\\
J_1 & = \left(\underline{2}13\underline{4} - \underline{2}31\underline{4}\right) + \left(13\underline{24} - 31\underline{24}\right) + \left(3\underline{24}1 - 1\underline{24}3\right) + \left(\underline{24}13 - \underline{24}31\right)\\
  & +\left(4\underline{23}1 - 1\underline{23}4\right) + \left(14\underline{23} - 41\underline{23}\right) +\left(\underline{23}14 - \underline{23}41\right) +\left(\underline{2}14\underline{3} - \underline{2}41\underline{3}\right)\\
K_1 & = \left(24\underline{13} - 42\underline{13}\right) + \left(4\underline{13}2 - 2\underline{13}4\right)+\left(\underline{1}24\underline{3} - \underline{1}42\underline{3}\right) + \left(\underline{13}24 - \underline{13}42\right)\\
  & +\left(4\underline{23}1 - 1\underline{23}4\right) + \left(14\underline{23} - 41\underline{23}\right) +\left(\underline{23}14 - \underline{23}41\right) +\left(\underline{2}14\underline{3} - \underline{2}41\underline{3}\right)
\end{align*}    
Further, let $H_2$, $I_2$, $J_2$ and $K_2$ be the corresponding elements of $\cA^{\tau_2}$ as in the proof of Theorem~\ref{thm:set8a}, and
let $M$ be the following (positive definite) matrix.
\[M=\begin{bmatrix}
35&0&12&0\\
0&35&0&-12\\
12&0&37&0\\
0&-12&0&37\\
\end{bmatrix} \]
It holds that
\[h_{\mu}\left(\unlab{w_1Mw_1^T}{\tau_1}+\unlab{w_2Mw_2^T}{\tau_2}\right)=16\left(\sum_{\pi\in\overline{S}}d(\pi,\mu)-\frac{2}{3}\right)\]
where $w_1=(H_1,I_1,J_1,K_1)$ and $w_2=(H_2,I_2,J_2,K_2)$.
This implies that
\[0\le\sum_{\pi\in\overline{S}}d(\pi,\mu)-\frac{2}{3}\]
and equality holds if and only if
both $h_{\mu}^{\tau_1}(w_1Mw_1^T)=0$ with probability one and
$h_{\mu}^{\tau_2}(w_2Mw_2^T)=0$ with probability one.
Since all the eigenvalues of $M$ are positive (the matrix has eigenvalues $36+\sqrt{145}$ and $36-\sqrt{145}$, each with multiplicity two),
$h_{\mu}^{\tau_1}(w_1Mw_1^T)=0$ if and only if
$h_{\mu}^{\tau_1}(H_1)=0$, $h_{\mu}^{\tau_1}(I_1)=0$, $h_{\mu}^{\tau_1}(J_1)=0$ and $h_{\mu}^{\tau_1}(K_1)=0$.
Hence, if equality holds, then $h_{\mu}^{\tau_1}(A_1)=0$ with probability one (note that $A_1=H_1-I_1$).
A symmetric argument yields that $h_{\mu}^{\tau_2}(A_2)=0$ with probability one.
The statement of the theorem now follows from Lemmas~\ref{lm:supp}, \ref{lm:A1} and \ref{lm:A2}.
\end{proof}

Finally, we prove the last main theorem of this section.

\begin{theorem}
\label{thm:set12}
Let $S=\{1234,1243,1432,2134,2143,2341,3214,3412,3421,4123$, $4312,4321\}$.
It holds that
\[\sum_{\pi\in S}d(\pi,\mu)\ge\frac{1}{2}\]
for every permuton $\mu$, and equality holds if and only if $\mu$ is uniform.
\end{theorem}

\begin{proof}
Consider the following four elements of $\cA^{\tau_1}$.
\begin{align*}
L_1 & =\left(4\underline{2}1\underline{3} - 1\underline{2}4\underline{3}\right) + \left(4\underline{1}2\underline{3} - 2\underline{1}4\underline{3}\right) + \left(\underline{2}3\underline{4}1 - \underline{2}1\underline{4}3\right) \\
  & + \left(4\underline{23}1 - 1\underline{23}4\right) + \left(\underline{1}2\underline{3}4 - \underline{1}4\underline{3}2\right) + \left(3\underline{24}1 - 1\underline{24}3\right)\\
M_1 & =\left(\underline{2}1\underline{3}4 - \underline{2}4\underline{3}1\right) + \left(1\underline{23}4 - 4\underline{23}1\right) +  \left(\underline{1}2\underline{3}4 - \underline{1}4\underline{3}2\right) \\
  & + \left(1\underline{24}3 - 3\underline{24}1\right) +\left(2\underline{13}4 - 4\underline{13}2\right) + \left(3\underline{24}1 - 1\underline{24}3\right)\\
N_1 & = \left(\underline{12}43 - \underline{12}34\right) + \left(21\underline{34} - 12\underline{34}\right) + \left(\underline{1}32\underline{4} - \underline{1}23\underline{4}\right) +\left(\underline{2}14\underline{3} - \underline{2}41\underline{3}\right)\\
  & + \left(2\underline{14}3 - 3\underline{14}2\right) + \left(\underline{23}14 - \underline{23}41\right) +\left(32\underline{14} - 23\underline{14}\right) + \left(\underline{1}43\underline{2} - \underline{1}34\underline{2}\right)\\
  & + \left(1\underline{34}2 - 2\underline{34}1\right) +\left(\underline{3}21\underline{4} - \underline{3}12\underline{4}\right) + \left(3\underline{12}4 - 4\underline{12}3\right) + \left(\underline{34}21 - \underline{34}12\right) \\
  & +\left(43\underline{12} - 34\underline{12}\right) + \left(\underline{14}32 - \underline{14}23\right) + \left(14\underline{23} - 41\underline{23}\right) + \left(4\underline{23}1 - 1\underline{23}4\right) \\
O_1 & = \left(\underline{1}42\underline{3} - \underline{1}24\underline{3}\right) + \left(\underline{13}42 - \underline{13}24\right) + \left(13\underline{24} - 31\underline{24}\right) +\left(\underline{24}13 - \underline{24}31\right)\\
  & + \left(42\underline{13} - 24\underline{13}\right) + \left(\underline{2}13\underline{4} - \underline{2}31\underline{4}\right) +\left(2\underline{13}4 - 4\underline{13}2\right) + \left(3\underline{24}1 - 1\underline{24}3\right)
\end{align*}
Further, let $L_2$, $M_2$, $N_2$ and $O_2$ be the corresponding elements of $\cA^{\tau_2}$ as in the proof of Theorem~\ref{thm:set8a}, and
let $M$ be the following (positive definite) matrix.
\[M=\begin{bmatrix}
1132&-652&-638& 197& 326\\
-652& 774& 516&-68 &-326\\
-638& 516& 774& 68 &-326\\
197 &-68 & 68 & 172& 0  \\
326 &-326&-326& 0  & 516
\end{bmatrix}\]
It holds that
\[h_{\mu}\left(\unlab{w_1Mw_1^T}{\tau_1}+\unlab{w_2Mw_2^T}{\tau_2}\right)=172\left(\sum_{\pi\in S}d(\pi,\mu)-\frac{1}{2}\right)\]
where $w_1=(A_1,L_1,M_1,N_1,O_1)$ and $w_2=(A_2,L_2,M_2,N_2,O_2)$.
This implies that
\[0\le\sum_{\pi\in S}d(\pi,\mu)-\frac{1}{2}\]
and equality holds if and only if
both $h_{\mu}^{\tau_1}(w_1Mw_1^T)=0$ with probability one and
$h_{\mu}^{\tau_2}(w_2Mw_2^T)=0$ with probability one.
Since all the eigenvalues of $M$ are positive (because all the leading principal minors of $M$ are positive),
it follows that equality holds if and only if both $h_{\mu}^{\tau_1}(A_1)=0$
with probability one and $h_{\mu}^{\tau_2}(A_2)=0$ with probability one.
The statement of the theorem now follows from Lemmas~\ref{lm:supp}, \ref{lm:A1} and \ref{lm:A2}.
\end{proof}

\section{Perturbations of the uniform permuton}
\label{sec:perturbe}

In this section, we analyze pattern densities in step permutons obtained from the uniform permuton by a perturbation.
This analysis will yield that most of the sets different from those listed in Theorem~\ref{thm:main}
are not $\Sigma$-forcing.

We start with the definition of a step permuton.
If $A$ is a (non-negative) doubly stochastic square matrix of order $n$,
i.e., each row sum and each column sum of $A$ is equal to one,
we can associate with it a permuton $\mu[A]$ by setting
\[\mu[A](X):=\sum_{i,j\in [n]} A_{ij}\cdot n\cdot
                              \left\lvert X\;\cap\;\left[\frac{i-1}{n},\frac{i}{n}\right)\times\left[\frac{j-1}{n},\frac{j}{n}\right)\right\rvert\]
for every Borel set $X\subseteq [0,1]^2$.
We refer to permutons that can be obtained in this way from a doubly stochastic square matrix as \emph{step permutons}.
A straightforward computation yields the following expression for the density of a $k$-permutation $\pi$ in $\mu[A]$;
we use $f:[k]\nearrow[n]$ to mean that $f$ is a non-decreasing function from $[k]$ to $[n]$.
Indeed, each of the summands corresponds to the probability that
the $\mu[A]$-random permutation of order $k$ is $\pi$ and
the $k$ points defining $\pi$ are sampled from the squares with coordinates $(f(i),g(\pi(i)))$, $i\in [n]$.
\begin{lemma}
\label{lm:dens}
Let $A$ be a doubly stochastic square matrix of order $n$, and $\pi$ a $k$-permutation.
It holds that
\[d(\pi,\mu[A])=\frac{k!}{n^k}\sum_{f,g:[k]\nearrow [n]}
  \frac{1}{\prod\limits_{i\in [n]}\lvert f^{-1}(i)\rvert!\cdot\lvert g^{-1}(i)\rvert!}\times\prod_{i\in [k]}A_{f(i),g(\pi(i))}.\]
\end{lemma}
\noindent For $i,j\in [n-1]$, let $B^{ij}$ be the matrix such that
\[B^{ij}_{i'j'}=
  \begin{cases}
  +1 & \mbox{if either $i'=i$ and $j'=j$ or $i'=i+1$ and $j'=j+1$,} \\
  -1 & \mbox{if either $i'=i$ and $j'=j+1$ or $i'=i+1$ and $j'=j$, and} \\
  0 & \mbox{otherwise.}
  \end{cases}
	\]
In the following exposition, the order $n$ of the matrices $B^{ij}$ will always be clear from the context and
so we use only the indices $i$ and $j$ to avoid unnecessarily complex notation.

For an integer $n$ and a permutation $\pi$, we define a function $h_{\pi,n}:U_n\to\RR$
on the cube $U_n:=\{\vec x\in \RR^{[n-1]^2}:\lVert \vec x\rVert_\infty \le 1/4n\}$ around the origin as
\[h_{\pi,n}(x_{1,1},\ldots,x_{n-1,n-1}):=d\left(\pi,\mu\left[A+\sum_{i,j\in [n-1]}x_{ij}B^{ij}\right]\right)\]
where $A$ is the $n\times n$ matrix with all entries equal to $1/n$. 
Note that this is well-defined as $A+\sum_{i,j\in [n-1]}x_{ij}B^{ij}$ is doubly stochastic 
whenever $x_{i,j}\in U_n$ for all $i,j\in [n-1]$.
More generally, if $S$ is a set of permutations, we define $h_{S,n}:U_n\to\RR$ as
\[h_{S,n}(\vec x):=\sum_{\pi\in S}h_{\pi,n}(\vec x)\;\mbox{.}\]
In this section, we are concerned with sets $S$ that consist of $4$-permutations only.

If $S$ is a set of $4$-permutations, we define the \emph{cover matrix} of $S$ to be a $4\times 4$ matrix $C^S$ such that
$C^S_{ij}$ is the number of permutations $\pi\in S$ such that $\pi(j)=i$.
If the set $S$ is clear from the context, then we will just write $C$ for the cover matrix.
We show that the gradient of $h_{S,n}$ at the origin is determined by the cover matrix of $S$.

\begin{lemma}
\label{lm:cover}
Let $n$ be an integer and $S$ a set of $4$-permutations with cover matrix $C$.
It holds that
\begin{align*}
 \frac{\partial}{\partial x_{ij}}h_{S,n}(0,\ldots,0) = & 
  \frac{4!}{n^7}\sum_{f,g:[4]\nearrow [n]}
  \frac{1}{\prod\limits_{m\in [n]}\lvert f^{-1}(m)\rvert!\cdot\lvert g^{-1}(m)\rvert!}\times
  \left(\sum_{\substack{k\in f^{-1}(i)\\\ell\in g^{-1}(j)}}C_{k,\ell}\right. \\
  & \left. -\sum_{\substack{k\in f^{-1}(i+1)\\\ell\in g^{-1}(j)}}C_{k,\ell}
             -\sum_{\substack{k\in f^{-1}(i)\\\ell\in g^{-1}(j+1)}}C_{k,\ell}
	     +\sum_{\substack{k\in f^{-1}(i+1)\\\ell\in g^{-1}(j+1)}}C_{k,\ell}\right)
\end{align*}
for every $i,j\in [n-1]$.
\end{lemma}

\begin{proof}
Since both the first derivative on the left hand side and 
the expression on the right hand side in the statement of the lemma
are additive with respect to adding elements of the set $S$,
it is enough to prove the lemma when $S$ contains a single element $\pi$.
In such case, the formula given in the statement of the lemma follows directly from Lemma~\ref{lm:dens}.
\end{proof}

Lemma~\ref{lm:cover} yields the following.

\begin{lemma}
\label{lm:easy}
Let $n$ be an integer and $S$ a set of $4$-permutations.
If the cover matrix $C$ is constant,
then the gradient
\[\nabla h_{S,n}(0,\ldots,0)=\left(\frac{\partial}{\partial x_{ij}}h_{S,n}(0,\ldots,0)\right)_{i,j\in [n-1]}\]
is zero.
\end{lemma}

\begin{proof}
We start by defining an operator on non-decreasing functions from $[4]$ to $[n]$. 
Given $f:[4]\nearrow [n]$ and an index $k\in[n-1]$, we define $\tilde f^{(k)}$ as follows. 
Let $Z$ be the image of $f$ viewed as a multiset with every $k$ replaced with $k+1$ and every $k+1$ replaced with $k$.
Then $\tilde f^{(k)}$ is the unique non-decreasing function from $[4]$ to $[n]$ whose image is $Z$.
Informally speaking, we switch the values $k$ and $k+1$ and reorder to obtain a non-decreasing function.
Note that $f=\widetilde{(\tilde f^{(k)})}^{\raisebox{-5pt}{\scriptsize $(k)$}}$ for all $f$ and $k$, and
$f=\tilde f^{(k)}$ if $\lvert f^{-1}(k)\rvert=\lvert f^{-1}(k+1)\rvert$.

We now analyze individual summands in the sum in the statement of Lemma~\ref{lm:cover}.
Fix two indices $i$ and $j$, and a function $g:[4]\nearrow [n]$.
If $f=\tilde f^{(i)}$, then the expression in the parenthesis evaluates to zero.
If $f\not=\tilde f^{(i)}$, then the expressions for $f$ and $\tilde f^{(i)}$ have opposite signs, in particular their contributions cancel out.
We conclude that the sum is equal to zero if all the entries of the cover matrix $C$ are the same.
The lemma now follows.
\end{proof}

Lemma~\ref{lm:cover} establishes that the gradient $\nabla h_{S,n}(0,\ldots,0)$ for a set $S$ of $4$-permu\-ta\-tions
is a linear function of the entries of the cover matrix $C$ of $S$.
Analyzing the matrix corresponding to this linear function for $n\in\{4,5\}$ yields that for such $n$
the gradient $\nabla h_{S,n}(0,\ldots,0)$ is zero if and only if the cover matrix of $S$ is constant.
Instead of providing this technical computation here,
we give a more illustrative proof that the converse of Lemma~\ref{lm:easy} holds for large enough integers $n$,
as this is sufficient for our exposition.

\begin{lemma}
\label{lm:large}
Let $S$ be a set of $4$-permutations whose cover matrix is not constant. Then there exists an integer $n$ such that the gradient
\[\nabla h_{S,n}(0,\ldots,0)=\left(\frac{\partial}{\partial x_{ij}}h_{S,n}(0,\ldots,0)\right)_{i,j\in [n-1]}\]
is non-zero.
\end{lemma}

\begin{proof}
We will assume that the gradient $\nabla h_{S,n}(0,\ldots,0)$ is zero and
establish that the entries of the cover matrix satisfy
$C_{k,\ell}-C_{k+1,\ell}-C_{k,\ell+1}+C_{k+1,\ell+1}=0$ for all $k,\ell\in [3]$.
We will then use this to show that the cover matrix $C$ must be constant.

We start by analyzing the partial derivative $\frac{\partial}{\partial x_{ij}}h_{S,n}(0,\ldots,0)$ for $i=1$ and $j=1$.
Recall the notation $\tilde f^{(k)}$ from the proof of Lemma~\ref{lm:easy}.
If $\lvert\im(f)\cap\{1,2\}\rvert\le 1$ or $\lvert\im(g)\cap\{1,2\}\rvert\le 1$,
then the summands in the expression given in Lemma~\ref{lm:cover}
corresponding to $(f,g)$, $(\tilde f^{(1)},g)$, $(f,\tilde g^{(1)})$ and $(\tilde f^{(1)},\tilde g^{(1)})$ sum to zero.
Hence, we need to focus on the summands where $\{1,2\}\subseteq\im(f)$ and $\{1,2\}\subseteq\im(g)$.
Note the number of summands such that $f$ or $g$ is not injective is $O(n^3)$,
which yields the following.
\begin{align*}
  \frac{\partial}{\partial x_{11}}h_{S,n}(0,\ldots,0) &= 
  \frac{4!}{n^7}\left(\sum_{\substack{f,g:[4]\nearrow [n]\\f(1)=1,f(2)=2,\lvert\im(f)\rvert=4\\g(1)=1,g(2)=2,\lvert\im(g)\rvert=4}}
                            \left(C_{11}-C_{12}-C_{21}+C_{22}\right)+O(n^3)\right) \\
  & = \frac{4!}{n^7}{n-2\choose 2}^2\left(C_{11}-C_{12}-C_{21}+C_{22}\right)+O\left(\frac{1}{n^4}\right).
\end{align*}
If $n$ is sufficiently large, the above expression can be zero only if $C_{11}-C_{12}-C_{21}+C_{22}=0$.
An analogous argument for $i=1$ and $j=n-1$ yields that $C_{13}-C_{14}-C_{23}+C_{24}=0$,
for $i=n-1$ and $j=1$ that $C_{31}-C_{32}-C_{41}+C_{42}=0$, and
for $i=n-1$ and $j=n-1$ that $C_{33}-C_{34}-C_{43}+C_{44}=0$.

We next analyze the partial derivative $\frac{\partial}{\partial x_{ij}}h_{S,n}(0,\ldots,0)$ for $i=1$ and $j=\lfloor n/2\rfloor$.
If $\lvert\im(f)\cap\{1,2\}\rvert\le 1$ or $\lvert\im(g)\cap\{\lfloor n/2\rfloor,\lfloor n/2\rfloor+1\}\rvert\le 1$,
then the summands in the expression given in Lemma~\ref{lm:cover}
corresponding to $(f,g)$, $(\tilde f^{(1)},g)$, $(f,\tilde g^{(\lfloor n/2\rfloor)})$ and $(\tilde f^{(1)},\tilde g^{(\lfloor n/2\rfloor)})$ sum to zero.
Hence, we need to focus on the summands where $\lvert\im(f)\cap\{1,2\}\rvert=2$ and
$\lvert\im(g)\cap\{\lfloor n/2\rfloor,\lfloor n/2\rfloor+1\}\rvert=2$.
Since the number of summands such that $f$ or $g$ is not injective is $O(n^3)$,
we obtain that
\begin{align*}
  \frac{\partial}{\partial x_{1,\lfloor n/2\rfloor}}h_{S,n}(0,\ldots,0)=&\frac{4!}{n^7}\left(
  \sum_{\substack{f,g:[4]\nearrow [n]\\f(1)=1,f(2)=2,\lvert\im(f)\rvert=4\\g(1)=\lfloor n/2\rfloor,g(2)=\lfloor n/2\rfloor+1,\lvert\im(g)\rvert=4}}\left(C_{11}-C_{12}-C_{21}+C_{22}\right)\right.\\
  &+\sum_{\substack{f,g:[4]\nearrow [n]\\f(1)=1,f(2)=2,\lvert\im(f)\rvert=4\\g(2)=\lfloor n/2\rfloor,g(3)=\lfloor n/2\rfloor+1,\lvert\im(g)\rvert=4}}\left(C_{12}-C_{13}-C_{22}+C_{23}\right)\\
  &\left.+\sum_{\substack{f,g:[4]\nearrow [n]\\f(1)=1,f(2)=2,\lvert\im(f)\rvert=4\\g(3)=\lfloor n/2\rfloor,g(4)=\lfloor n/2\rfloor+1,\lvert\im(g)\rvert=4}}\left(C_{13}-C_{14}-C_{23}+C_{24}\right)\right)\\
  &+O\left(\frac{1}{n^4}\right).
\end{align*}
Since the first and the third sum are equal to zero, we obtain that
\[\frac{\partial}{\partial x_{1,\lfloor n/2\rfloor}}h_{S,n}(0,\ldots,0)=\left(C_{12}-C_{13}-C_{22}+C_{23}\right)\cdot\Theta\left(\frac{1}{n^3}\right)+O\left(\frac{1}{n^4}\right).\]
Hence, if $n$ is large enough and this partial derivative is zero, it must hold that $C_{12}-C_{13}-C_{22}+C_{23}=0$.
An analogous argument for $i=\lfloor n/2\rfloor$ and $j=1$ yields that $C_{21}-C_{22}-C_{31}+C_{32}=0$,
for $i=n-1$ and $j=\lfloor n/2\rfloor$ that $C_{32}-C_{33}-C_{42}+C_{43}=0$, and
for $i=\lfloor n/2\rfloor$ and $j=n-1$ that $C_{23}-C_{24}-C_{33}+C_{34}=0$.

Finally, we analyze the partial derivative $\frac{\partial}{\partial x_{ij}}h_{S,n}(0,\ldots,0)$ for $i=j=\lfloor n/2\rfloor$.
As in the preceding two cases, we consider the functions $\tilde f^{(\lfloor n/2\rfloor)}$ and $\tilde g^{(\lfloor n/2\rfloor)}$ to conclude that the summands with $\lvert\im(f)\cap\{\lfloor n/2\rfloor,\lfloor n/2\rfloor+1\}\rvert\le 1$ or
$\lvert\im(g)\cap\{\lfloor n/2\rfloor,\lfloor n/2\rfloor+1\}\rvert\le 1$ sum to zero.
We next express the partial derivative as the sum of nine terms corresponding to injective mappings $f$ and $g$
with $\{\lfloor n/2\rfloor,\lfloor n/2\rfloor+1\}\subseteq\im(f)$ and
$\{\lfloor n/2\rfloor,\lfloor n/2\rfloor+1\}\subseteq\im(g)$ (the terms are determined by the preimages of $\lfloor n/2\rfloor$ and
$\lfloor n/2\rfloor+1$).
Eight of these terms correspond to the sums of the entries of the cover matrix that we have already shown to be zero,
which leads to the following expression for the considered partial derivative:
\[\frac{\partial}{\partial x_{\lfloor n/2\rfloor,\lfloor n/2\rfloor}}h_{S,n}(0,\ldots,0)=\left(C_{22}-C_{23}-C_{32}+C_{33}\right)\cdot\Theta\left(\frac{1}{n^3}\right)+O\left(\frac{1}{n^4}\right).\]
Hence, if $n$ is large enough and the partial derivative is zero, it must hold that $C_{22}-C_{23}-C_{32}+C_{33}=0$.

Since the cover matrix $C$ satisfies that $C_{k,\ell}-C_{k,\ell+1}-C_{k+1,\ell}+C_{k+1,\ell+1}=0$ for all $k,\ell\in [3]$, $C$ is of the form
\[C=\begin{pmatrix}
    a & b & c & d \\
    e & b+e-a & c+e-a & d+e-a \\
    f & b+f-a & c+f-a & d+f-a \\
    g & b+g-a & c+g-a & d+g-a \\
    \end{pmatrix}\]
for some integers $a,\ldots,g$. Since $C$ is a cover matrix for a set $S$ of $4$-permutations, each row and each column must sum to $\lvert S\rvert$,
i.e., the sums of the entries of each row are equal and the same holds for the columns of $C$.
It follows that $b=c=d$ and $e=f=g$, so
\[C=\begin{pmatrix}
    a & b & b & b \\
    e & b+e-a & b+e-a & b+e-a \\
    e & b+e-a & b+e-a & b+e-a \\
    e & b+e-a & b+e-a & b+e-a \\
    \end{pmatrix}.\]
It now follows that $b=e$ (otherwise, the sum of the second row and the second column would differ),
which yields that the matrix $C$ must be of the form
\[C=\begin{pmatrix}
    a & b & b & b \\
    b & 2b-a & 2b-a & 2b-a \\
    b & 2b-a & 2b-a & 2b-a \\
    b & 2b-a & 2b-a & 2b-a \\
    \end{pmatrix}.\]
Hence, we get that $a+3b=7b-3a$, which yields that $a=b$.
We conclude that the matrix $C$ is constant.
\end{proof}

The following lemma will be used to analyze sets of $4$-permutations with constant cover matrix.

\begin{lemma}
\label{lm:second}
Let $S$ be a set of $4$-permutations such that the cover matrix $C$ is constant.
The Hessian matrix of the second order partial derivatives of $h_{S,5}$ at $(0,\ldots,0)$
has both a positive and a negative eigenvalue,
unless $S$ is symmetric to one of the following sets of $4$-permutations
\begin{itemize}
\item $\{1234,2143,3412,4321\}$,
\item $\{1234,1243,2134,2143,3412,3421,4312,4321\}$,
\item $\{1234,1432,2143,2341,3214,3412,4123,4321\}$,
\item $\{1324,1342,2413,2431,3124,3142,4213,4231\}$,
\item $\{1342,1423,2314,2431,3124,3241,4132,4213\}$,
\item $\{1234,1243,1324,2134,2143,2413,3142,3412,3421,4231,4312,4321\}$,
\item $\{1234,1243,1342,2134,2143,2431,3124,3412,3421,4213,4312,4321\}$,
\item $\{1234,1243,1342,2134,2143,2431,3214,3412,3421,4123,4312,4321\}$,
\item $\{1234,1243,1432,2134,2143,2341,3214,3412,3421,4123,4312,4321\}$,
\item $\{1234,1243,1432,2134,2341,2413,3142,3214,3421,4123,4312,4321\}$,
\item $\{1234,1243,1432,2143,2314,2341,3214,3412,3421,4123,4132,4321\}$,
\item $\{1234,1342,1423,2143,2314,2431,3124,3241,3412,4132,4213,4321\}$,
\item $\{1234,1342,1423,2314,2413,2431,3124,3142,3241,4132,4213,4321\}$, or
\end{itemize}
to the complement of one of them.
\end{lemma}

\begin{proof}
For a $4$-permutation $\pi$, let $H_{\pi}$ be the Hessian matrix (of order sixteen)
\[\left(\frac{\partial^2}{\partial x_{ij}\partial x_{i'j'}}h_{\{\pi\},5}(0,\ldots,0)\right)_{i,j,i',j'\in [4]}.\]
The matrices $H_{\pi}$ for all $4$-permutations can be found in Appendix 1.
For a set $S$ of $4$-permutations, let $H_S$ be the corresponding Hessian matrix,
i.e.,
\[H_S=\sum_{\pi\in S}H_{\pi}.\]
Note that $H_S=-H_{\overline{S}}$ where $\overline{S}$ is the complement of $S$ with respect to the set of all $4$-permutations.
If the cover matrix of $S$ is constant, then $|S|$ must be divisible by four.
Up to symmetry, there are $12$ sets $S$ with $4$ elements and $65$ sets $S$ with $8$ elements whose cover matrix is constant.
Up to symmetry and taking complements, there are $68$ sets $S$ with $12$ elements whose cover matrix is constant.
These sets are listed in Appendices 2--4 together with the corresponding matrices $H_S$ and
their largest and smallest eigenvalues.
An inspection of these values yields the statement of the lemma (the sets $S$ such that the matrix $H_S$ does not have
both positive and negative eigenvalues are highlighted by the bold font in Appendices 2--4).
\end{proof}

We are now ready to prove the main theorem of this section.

\begin{theorem}
\label{thm:perturbe}
Let $S$ be a set of $4$-permutations.
There exists an integer $n$ and
$\vec x,\vec y\in U_n$ such that $h_{S,n}(\vec x)<\lvert S\rvert/24<h_{S,n}(\vec y)$,
unless $S$ is symmetric to one of the sets of $4$-permutations listed in Lemma~\ref{lm:second}, or
to the complement of one of them.
\end{theorem}

\begin{proof}
If the cover matrix $C$ of $S$ is not constant,
then there exists an integer $n$ such that the gradient $\nabla h_{S,n}(0,\ldots,0)$ is non-zero by Lemma~\ref{lm:large}.
Hence, we can set $\vec x=-\varepsilon \nabla h_{S,n}(0,\ldots,0)$ and $\vec y=\varepsilon \nabla h_{S,n}(0,\ldots,0)$
for a sufficiently small positive $\varepsilon$.
If the cover matrix $C$ of $S$ is constant, then $\nabla h_{S,n}(0,\ldots,0)$ is zero for every integer $n$ by Lemma~\ref{lm:easy},
in particular, for $n=5$.
However, unless $S$ is symmetric to one of the sets of $4$-permutations listed in Lemma~\ref{lm:second} or
to the complement of one of them,
the Hessian matrix of the second partial derivatives of $h_{S,5}$ at $(0,\ldots,0)$ has both positive and negative eigenvalues.
Hence,
we can set $\vec x$ to be an $\varepsilon$-multiple of the eigenvector corresponding to a negative eigenvalue of the Hessian matrix and
$\vec y$ to be an $\varepsilon$-multiple of the eigenvector corresponding to a positive eigenvalue for a sufficiently small positive $\varepsilon$.
\end{proof}

\section{Non-$\Sigma$-forcing sets}
\label{sec:examples}

We start this section with a lemma which asserts that
in order to show that a set $S$ of $4$-permutations is not $\Sigma$-forcing,
it is enough to find a permuton where the sum of pattern densities is smaller than $\lvert S\rvert/24$, and
a permuton where the sum of pattern densities is larger than $\lvert S\rvert/24$.

\begin{lemma}
\label{lm:examples}
Let $S$ be a set of $4$-permutations.
If there exist permutons $\mu_1$ and $\mu_2$ such that
\[\sum_{\pi\in S}d(\pi,\mu_1)<\frac{\lvert S\rvert}{24}\mbox{ and }\sum_{\pi\in S}d(\pi,\mu_2)>\frac{\lvert S\rvert}{24},\]
then there exists a non-uniform permuton $\mu$ such that
\[\sum_{\pi\in S}d(\pi,\mu)=\frac{\lvert S\rvert}{24}.\]
\end{lemma}

\begin{figure}
\begin{center}
\epsfbox{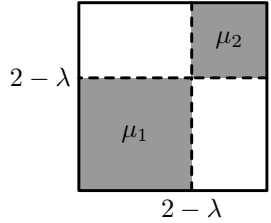}
\end{center}
\caption{The permuton $\mu_\lambda$ in the proof of Lemma~\ref{lm:examples}. The support of the permuton lies in the gray area.}
\label{fig:lambda}
\end{figure}

\begin{proof}
Define a permuton $\mu_\lambda$ for $\lambda\in (1,2)$ as follows:
\begin{align*}
\mu_\lambda(X) & =(2-\lambda)\cdot\mu_1\left(\frac{1}{2-\lambda}\times\left(X\cap [0,2-\lambda]^2\right)\right)\\
               & +(\lambda-1)\cdot\mu_2\left(\frac{1}{\lambda-1}\times\left(X\cap [2-\lambda,1]^2-(2-\lambda,2-\lambda)\right)\right),
\end{align*}	       
where $\alpha\times X$ stands for $\{\alpha\cdot x, x\in X\}$ and $X-v$ for $\{x-v,x\in X\}$.
The definition of a permuton $\mu_\lambda$ is illustrated in Figure~\ref{fig:lambda}.
Note that $\mu_\lambda$ is $\mu_1$ for $\lambda=1$ and $\mu_2$ for $\lambda=2$.
Next define a function $f:[1,2]\to[0,1]$ as
\[f(\lambda)=\sum_{\pi\in S}d(\pi,\mu_\lambda).\]
Observe that $f$ is a continuous function on the interval $[1,2]$.
Hence, there exists $\lambda\in (1,2)$ such that $f(\lambda)=\lvert S\rvert/24$.
Since the permuton $\mu_\lambda$ is not uniform for any $\lambda\in(1,2)$, the statement of the lemma follows.
\end{proof}

We are now ready to prove the main theorem of this section.

\begin{theorem}
\label{thm:examples}
Let $S$ be a set of $4$-permutations. Then there exists a non-uniform permuton $\mu$ such that
\[\sum_{\pi\in S}d(\pi,\mu)=\frac{\lvert S\rvert}{24}\]
unless the set $S$ is one of the following sets of $4$-permutations
\begin{itemize}
\item $\{1234,1243,2134,2143,3412,3421,4312,4321\}$,
\item $\{1234,1432,2143,2341,3214,3412,4123,4321\}$,
\item $\{1324,1342,2413,2431,3124,3142,4213,4231\}$,
\item $\{1324,1423,2314,2413,3142,3241,4132,4231\}$,
\item $\{1234,1243,1432,2134,2143,2341,3214,3412,3421,4123,4312,4321\}$, or
\end{itemize}
the complement of one of them.
\end{theorem}

\begin{proof}
Fix a set $S$ of $4$-permutations that is not one of the sets listed in the statement of the lemma.
We can assume that $\lvert S\rvert\le 12$ by considering the complement of $S$ if necessary.
By Lemma~\ref{lm:examples}, it suffices to find permutons $\mu_1$ and $\mu_2$ such that
the sum of the pattern densities of the permutations contained in $S$ for $\mu_1$
is less than $\lvert S\rvert/24$ and
for $\mu_2$ is larger than $\lvert S\rvert/24$.
If $S$ is not symmetric to a set listed in the statement of Lemma~\ref{lm:second},
such permutons $\mu_1$ and $\mu_2$ exist by Theorem~\ref{thm:perturbe}.
Hence, we can assume that $S$ is one of the $9$ sets listed in the statement of Lemma~\ref{lm:second} (out of the $13$ sets listed in total) that 
are not symmetric to one of the five sets given in the statement of Theorem~\ref{thm:examples} (note that
the third and fourth sets listed in the statement of the theorem are symmetric).

We first consider the case $S=\{1342,1423,2314,2431,3124,3241,4132,4213\}$.
We choose $\mu_1$ to be the \emph{monotone increasing permuton},
i.e., the unique permuton such that $\supp{\mu_1}=\{(x,x),x\in [0,1]\}$.
The density of a pattern $\pi$ in $\mu_1$ is $1$ if $\pi$ is increasing and $0$ otherwise;
in particular, the sum of the pattern densities of the permutations from $S$ is zero.
Next, consider the following doubly stochastic matrix $A$
\[A=\begin{pmatrix}
    0 & 0 & 0 & 1 & 0 & 0 \\
    0 & 0 & 0 & 0 & 1 & 0 \\
    1 & 0 & 0 & 0 & 0 & 0 \\
    0 & 0 & 1 & 0 & 0 & 0 \\
    0 & 0 & 0 & 0 & 0 & 1 \\
    0 & 1 & 0 & 0 & 0 & 0 \\
    \end{pmatrix},\]
and set $\mu_2=\mu[A]$.
A direct computation yields that the sum of the pattern densities of the permutation contained in $S$ in $\mu_2$
is $\frac{25}{72}>\frac{1}{3}$.

Each of the eight sets $S$ that remain to be considered contain the permutation $1234$.
Hence, we set $\mu_2$ to be the monotone increasing permuton.
The permutons $\mu_1$ for these sets can be chosen as step permutons corresponding
to doubly stochastic matrices listed in Appendix 5.
\end{proof}

\section*{Acknowledgements}

The last author would like to thank Hays Whitlatch for bringing the results of Eric Zhang to his attention.

\bibliographystyle{bibstyle}
\bibliography{qperm8}
\end{document}